\documentclass{article}
\usepackage{graphicx}
\usepackage[english]{babel}
\usepackage[letterpaper,top=2cm,bottom=2cm,left=3cm,right=3cm,marginparwidth=1.75cm]{geometry}
\usepackage{amsmath}
\usepackage{graphicx}
\usepackage[colorlinks=true, allcolors=blue]{hyperref}
\usepackage{amsthm}
\usepackage[utf8]{inputenc}

\newtheorem{theorem}{Theorem}[section]
\newtheorem{lemma}[theorem]{Lemma}
\newtheorem{corollary}[theorem]{Corollary}
\newtheorem{proposition}[theorem]{Proposition}
\newtheorem{construction}[theorem]{Construction}
\newtheorem{conjecture}[theorem]{Conjecture}

\theoremstyle{definition}

\newtheorem*{example}{Example}

\newtheorem*{proposition*}{Proposition 2.4}

\theoremstyle{remark}
\newtheorem*{remark}{Remark}

\usepackage{amsfonts}
\usepackage{enumerate}  
\usepackage{array}
\usepackage{ulem}
\usepackage{csquotes}
\usepackage[dvipsnames]{xcolor}
\usepackage{comment}
\usepackage{titlesec}
\titlelabel{\thetitle.\quad}
\usepackage{ytableau}
\usepackage{adjustbox}
\usepackage{tikz}
\usetikzlibrary{shapes, positioning}
\usepackage{authblk}

\usepackage[symbol]{footmisc}

\begin{document}

\Large
 \begin{center}
On the Independence Numbers of the Cyclic Van der Waerden Hypergraphs\footnote{This work was supported by NSF DMS grant no. 1950563.} 

\hspace{10pt}

\large
Benjamin Liber

\hspace{10pt}

\small  
Drexel University \\
bl839@drexel.edu\\

\end{center}

\begin{abstract}
    Building upon the work of Berglund (2018), we establish a method for constructing subsets $B \subseteq \mathbb{Z}_{mk}$ such that $B$ does not contain any $k$-term cyclic arithmetic progressions mod $mk$, where $m,k \in \mathbb{Z}^+$ with $k \geq 3$. This construction thereby provides concrete lower bounds for the maximum size of such subsets. Additionally, it allows us to tightly bound specific chromatic numbers $\chi(mk,k)$ of $\mathbb{Z}_{mk}$ and helps increase the lower bounds of certain cyclic Van der Waerden numbers $W_{c}(k,r)$, originally introduced by Burkert and Johnson (2011) as a way of bounding the standard Van der Waerden numbers $W(k,r)$ from below for $r \geq 2$.
\end{abstract}

\normalsize
\color{black}
\section{Introduction}

\subsection{Notation and Terminology}

A \textit{hypergraph} is a pair $(V,E) = H$, where $V$ is a nonempty set and $E$ is a set of subsets of $V$. The elements $v \in V$ are called the \textit{vertices} of the hypergraph, and the elements $e \in E$ are called the \textit{hyperedges} (henceforth just \textit{edges}) of the hypergraph. 

A set $S \subseteq V$ is \textit{independent} in $H = (V,E)$ if and only if $S$ contains no edges $e \in E$. Correspondingly, the \textit{independence number} of $H$ is the size of the largest independent set in $H$. 

A \textit{j-coloring} of a hypergraph $H = (V,E)$ is a map $c : V \to \{1,2,\ldots,j\}$. An edge $e \in E$ is \textit{monochromatic}, with respect to coloring $c$, if there exists color $i \in \{1,2,\ldots,j\}$ such that $c(v) = i$ for all $v \in E$. That is, all vertices $v \in E$ are assigned the same color. A $j$-coloring of $H$ is said to be \textit{proper} if there exist no monochromatic edges of cardinality at least two. Correspondingly, the \textit{chromatic number} $\chi(H)$ of $H$ is defined to be the smallest number of colors necessary to define a proper coloring of $H$. Note that with a proper coloring of $H$, all vertices assigned the same color form an independent set. Hence, properly coloring the vertices in $V$ is equivalent to partitioning $V$ into independent sets.

Next, let $k,N, d\in \mathbb{Z}^+$ with $k \geq 3$\footnote{Note that we must have $k \geq 3$, for any two consecutive elements of a $k$-subset of $\mathbb{Z}_{N}$ automatically form a two-term cyclic arithmetic progression.}. A \textit{$k$-term cyclic arithmetic progression mod $N$} with \textit{common difference} $d$ is a $k$-element subset of $\mathbb{Z}_{N} := \mathbb{Z} / N\mathbb{Z}$, where each element is of the form $t+id$ for $i \in \{0,1,\ldots,k-1\}$. We say that $t$ is the \textit{base} of the progression and $k$ is the \textit{length}. It should be noted that, although we will use parenthesis $(\hspace{0.3em} )$ to denote such cyclic arithmetic progressions in order to avoid confusion with sets of integers denoted with brackets $\{ \hspace{0.3em}\}$, these progressions are sets themselves, not sequences. Hence, different bases and common differences may give the same cyclic arithmetic progression.

\begin{example}
\label{ex1}
    The $6$-term cyclic arithmetic progression $(0,2,4,6,8,10)$ in $\mathbb{Z}_{12}$ may have any of $0,2,4,6,8,$ or $10$ as a base and, for each choice of base, any of $2,4,8,$ or $10$ as a common difference.

   Similarly, the $5$-term cyclic arithmetic progression $(2,4,6,8,10)$ in $\mathbb{Z}_{12}$ can have base $2$ or $10$ with common difference $2$ or $10$, respectively.
\end{example}

As stated by Berglund \cite{Berglund}, if $d \in \{1,\ldots,N-1\}$ is a common difference of a cyclic arithmetic progression mod $N$, then so is $N-d$. Hence, there always exists a common difference $d$ such that $0 < d < \frac{N}{2}$. We refer to the smallest such positive integer $d$ as \textit{the} common difference of a progression. For example, the common difference of both progressions in the above example is $2$.

Note that $k \geq 3$ implies that when $N$ is even, $\frac{N}{2}$ cannot be a common difference for any $k$-term cyclic arithmetic progression mod $N$, since such progressions must have $k \geq 3$ distinct elements.

\subsection{Cyclic Van der Waerden Hypergraphs \& Numbers}\label{sec:1.2}
Similar to a coloring of a hypergraph, a \textit{$j$-coloring} of an arbitrary set $S$ is a map $c : S \to \{1,\cdots,j\}$. A subset $T \subseteq S$ is \textit{monochromatic}, with respect to coloring $c$, if there exists $i \in \{1,2,\ldots,j\}$ such that $c(t) = i$ for all $t \in T$. That is, all elements $t \in T$ are assigned the same color.

We say that an arbitrary set $S \subseteq \mathbb{Z}$ \textit{contains} a $k$-term (cyclic) arithmetic progression 
if there exist $k$ elements in $S$ that form a $k$-term progression. For reasons that will become clear momentarily, this can be rephrased as follows: $S$ \textit{contains} a $k$-term (cyclic) arithmetic progression if there exists a $k$-element subset $T \subseteq S$ such that the elements of $T$ form a $k$-term progression.

For $k,r \in \mathbb{Z}$ such that $k \geq 3$ and $r \geq 2$, the \textit{Van der Waerden Number} $W(k,r)$ is defined to be the smallest positive integer $N$ such that, for any coloring of $\{0, 1,\ldots, N-1\}$ with $r$ or fewer colors, there exists a monochromatic $k$-term arithmetic progression in $\{0,1,\ldots,N-1\}$. That is, for any coloring of $\{0,1,\ldots,N-1\}$ with $r$ or fewer colors, there exists a $k$-element monochromatic subset of $\{0,1,\ldots,N-1\}$ whose elements form a $k$-term arithmetic progression.

Van der Waerden \cite{Waerden} showed that $W(k,r)$ always exists, but finding $W(k,r)$ for various integers $k$ and $r$ is difficult. As a strategy for handling this problem, an extension of $W(k,r)$ called the \textit{cyclic Van der Waerden Numbers} $W_{c}(k,r)$ was introduced by Burkert and Johnson \cite{Johnson}. $W_{c}(k,r)$ is defined to be the smallest modulus $N$ such that for all positive integers $M \geq N$, if $\mathbb{Z}_{M}$ is $r$-colored, then there exists a monochromatic $k$-term cyclic arithmetic progression mod $M$ in $\mathbb{Z}_{M}$. That is, if $\mathbb{Z}_{M}$ is $r$-colored, then there exists a $k$-term monochromatic subset of $\mathbb{Z}_{M}$ whose elements form a $k$-term cyclic arithmetic progression mod $M$. 

Correspondingly, the \textit{chromatic number} $\chi(N,k)$ of $\mathbb{Z}_N$, for $k,N \in \mathbb{Z}^+$, is the minimum number of subsets necessary to partition $\mathbb{Z}_{N}$ such that no subset contains any $k$-term cyclic arithmetic progressions mod $N$. It is clear from the definitions if $\chi(N,k) = r$, then $W_c(k,r) > N$.

As described by Johnson and Yang \cite{Yang}, the \textit{cyclic Van der Waerden hypergraph} $H_{N,k} = (V_{N},E_{N,k})$ has vertex set $V_{N} = \{0,1,\ldots,N-1\} = \mathbb{Z}_{N}$ and edge set \\ $E_{N,k} =\{k\text{-term cyclic arithmetic progressions mod }N\}$.

Using notation from \cite{Berglund}, let $b(N,k)$ denote the maximum size of a subset $B \subseteq \mathbb{Z}_{N}$ such that $B$ contains no $k$-term cyclic arithmetic progressions mod $N$. In \cite{Johnson}, finding values of $b(N,k)$ was considered to be of possible use bounding the numbers $W_{c}(k,r)$ from below, since $b(N,k)$ is the independence number of $H_{N,k}$. 

\subsection{Main Goal}

Let $D(N,k)$ be the set of possible values of $\gcd(d,k)$, where $d$ ranges over all possible common differences of $k$-term arithmetic progressions mod $N$. Berglund \cite{Berglund} proved the following two results:

\begin{proposition}
    Let $m > 0$. Then, $b(mk,k) \leq mk-m$. If $D(mk,k) = \{1\}$, then $b(mk,k) = mk-m$.
\end{proposition}

\begin{proposition}
    $b(2k,k) = \begin{cases}
        2k-2 & \text{if $D(2k,k) = \{1\}$} \\
        2k-3 & \text{if $D(2k,k) = \{1,2\}$}.
    \end{cases}$
\end{proposition}

\begin{remark}
    As we shall see with Lemma \ref{prop:2.1}, $D(2k,k) = \{1\}$ simply means that $k$ is odd while $D(2k,k) = \{1,2\}$ means that $k$ is even.
\end{remark}

In this work, we extend Berglund's methods, which resolved the case $m \leq 2$, to study $b(mk,k)$ for arbitrary $m \in \mathbb{Z}^+$; equivalently, we investigate $b(N,k)$ when $N$ is a multiple of $k$. 

While Berglund's results for $m \leq 2$ yield exact formulas for $b(mk,k)$, the situation for $m > 2$ is more involved, and only tight upper and lower bounds are attainable. In this work, we aim to prove the following result for $m,k \in \mathbb{Z}^+, k \geq 3$:

\begin{theorem}
\label{maintheorem}
    Suppose $D(mk,k) = \{d_0,d_1,\ldots,d_{j}\}$ for some $j \geq 0$, and let $d_{-1} := 0$. Then \begin{equation*}
        mk-m \geq b(mk,k) \geq mk-\sum_{i=0}^{j}(d_i-d_{i-1})(m-d_i+1).
    \end{equation*}
\end{theorem}

To prove Theorem \ref{maintheorem}, we will construct specific subsets $B \subseteq \mathbb{Z}_{mk}$ that contain no $k$-term cyclic arithmetic progressions mod $mk$. Consequently, the sizes $|B|$ provide lower bounds for $b(mk,k)$ across various values of $m$ and $k$. 

In Section \ref{sec: lemmas}, we demonstrate how to restrict the number of progressions that must be avoided in $B$, thereby reducing the complexity of bounding $b(mk,k)$. In Section \ref{sec:3}, we describe how to construct such sets $B$ -- more precisely, subsets $F \subseteq\mathbb{Z}_{mk}$ so that $B := \mathbb{Z}_{mk} \setminus F$ -- and use these constructions to prove Theorem \ref{maintheorem}. Then, in Section \ref{sec:implications}, we use our bounds of $b(mk,k)$ to obtain tight upper and lower bounds on the chromatic numbers $\chi(mk,k)$, which in turn yield new lower bounds for the cyclic Van der Waerden numbers $W_{c}(k,r)$ for various $r \geq 2$. Finally, in Section \ref{sec:future}, we highlight possible directions for future work.

\section{Common Differences to Consider}
\label{sec: lemmas}

To bound $b(mk,k)$, it suffices to construct a set $B \subseteq \mathbb{Z}_{mk}$ that contains no $k$-term cyclic arithmetic progressions mod $mk$. To do so, we first establish a method for checking if some arbitrary set $S \subseteq \mathbb{Z}_{mk}$ contains any such progressions. 
Trivially, we can list out all such progressions and see if $S$ contains any of them. However, this is time-consuming and inefficient, especially for large $m$ and $k$. Rather, we can limit the amount of progressions to check for based on common differences. 

\subsection{Some General Lemmas}

For $d \in \mathbb{Z}_{N} \setminus \{0\}$, where $N \in \mathbb{Z}^{+}$, let $\langle d \rangle$ denote the subgroup of $\mathbb{Z}_{N}$ generated by $d$. The following is well-known, but we include a proof of it for the convenience of the reader:

\begin{lemma}
\label{lem:1.1}
    Let $N \in \mathbb{Z}^{+}$ and $d \in \mathbb{Z}_{N} \setminus \{0\}$. Then $|\langle d \rangle| = \frac{N}{gcd(d,N)}$.
\end{lemma}

\begin{proof}
    Let $g := \gcd(d,N)$. Then there exists $\lambda,\mu \in \mathbb{Z}$ such that $\lambda d + \mu N = g$. So, $\lambda d \equiv g \mod N$. Hence, $g \in \langle d \rangle \implies \langle g \rangle \subseteq \langle d \rangle$. On the other hand, $g$ divides $d$ in $\mathbb{Z}$, so $\langle d \rangle \subseteq \langle g \rangle$ in $\mathbb{Z}_{N}$. Therefore, $\langle g \rangle = \langle d \rangle$. Since $g$ divides $N$ in $\mathbb{Z}$, we have $|\langle d\rangle| = |\langle g \rangle | = \frac{N}{g} = \frac{N}{\gcd(d,N)}$.
\end{proof}

We are now able to prove the following useful lemma. The forward direction was already proven in \cite{Berglund} but will be proven again here along with the backward direction:

\begin{lemma}
\label{lem:1.2}
    Suppose $k, N \in \mathbb{Z}^{+}$, $k \geq 3$, and there exists a $k$-term cyclic arithmetic progression mod $N$ with common difference $d$. Then $k \leq \frac{N}{gcd(N,d)}$. Conversely, if $k, N, d \in \mathbb{Z}^{+}, d < \frac{N}{2}$, and $3 \leq k \leq \frac{N}{gcd(N,d)}$, then there exists a $k$-term cyclic arithmetic progression mod $N$ with common difference $d$.
\end{lemma}

\begin{proof}
    Let $k,N \in \mathbb{Z}^{+}$, $k \geq 3$, and suppose
    $A = (t,t+d,\ldots,t+(k-1)d)$ be a $k$-term cyclic arithmetic progression mod $N$ with common difference $d$. Then A is a translation by $t$, in $\mathbb{Z}_{N}$, of $(0,d,\ldots,(k-1)d)$; and $(0,d,\ldots,(k-1)d)$ is a $k$-element subset of $\langle d \rangle$. Hence, $|A| = k \leq |\langle d \rangle | = \frac{N}{\gcd(N,d)}$ by Lemma \ref{lem:1.1}.

    On the other hand, let $k,N,d \in \mathbb{Z}^{+}$, $d < \frac{N}{2}$, and suppose $3 \leq k \leq \frac{N}{\gcd(N,d)}$. Then $(0,d,\ldots,(k-1)d)$ is a $k$-element subset of $\mathbb{Z}_{N}$ and also a $k$-term cyclic arithmetic progression mod $N$ with common difference $d$.
\end{proof}

Recall that $D(N,k)$ is defined to be the set of possible values of $\gcd(d,k)$, where $d$ ranges over all possible common differences of $k$-term arithmetic progressions mod $N$. In order to check if an arbitrary set $S \subseteq \mathbb{Z}_{N}$ contains any $k$-term cyclic arithmetic progressions mod $N$, Lemma \ref{lem:1.2} shows it suffices to check only for progressions with common differences $d \in D(N,k)$.

\subsection{Specific Case of  $N=mk$}
We are ready to focus our attention to the case of $N = mk$ for $m,k \in \mathbb{Z}^+$ such that $k \geq 3$. We have the following:

\begin{lemma}
\label{prop:2.1}
    $D(mk,k) = \{g \in \mathbb{Z}^{+} : g \leq m, \ g \mid k\}$
\end{lemma}

\begin{proof}
    First, consider a positive integer $g$ such that $g \leq m$ and $g \mid k$. Then $0, g, \ldots, (k-1)g$ are all distinct mod $mk$ and thus form a $k$-term cyclic arithmetic progression mod $mk$ with common difference $g$. So, $\text{gcd}(g,k) = g \in D(mk,k)$.

    On the other hand, suppose $d$ is the common difference of some $k$-term cyclic arithmetic progression mod $mk$ labeled A. Let $\langle d \rangle$ denote the cyclic subgroup of $\mathbb{Z}_{mk}$ generated by $d$. Since $A$ is a translate of $k$ distinct elements of $\langle d \rangle$, $k \leq |\langle d \rangle | = \frac{mk}{\text{gcd}(mk,d)}$. Hence, $g := \text{gcd}(k,d) \leq \text{gcd}(mk,d) \leq \frac{mk}{k} = m$. Thus, $g$, an arbitrary member of $D(mk,k)$, is a positive integer satisfying $g \leq m$ and $g \mid k$.
\end{proof}

\begin{remark}
    For integers $N,k$ satisfying $3 \leq k \leq N$, we note that $(0, 1, \ldots, k-1)$ is a $k$-term cyclic arithmetic progression mod $N$. Hence, we always have $1 \in D(N,k)$. This conclusion also follows from Lemma \ref{prop:2.1}, where $N$ is a multiple of $k$. On the other hand, the five-term cyclic arithmetic progression $(2,4,6,8,10)$ in $\mathbb{Z}_{12}$ shows that when $N$ is not a multiple of $k$, there may be elements of $D(N,k)$ which do not divide $k$.
\end{remark}

Recall that the following was proven in \cite{Berglund}:

\begin{proposition}
 \label{prop:1.1}
$b(mk,k) \leq mk-m$. If $D(mk,k) = \{1\}$, then $b(mk,k) = mk-m$.
\end{proposition}

Going forth, we focus on constructing sets $B \subseteq \mathbb{Z}_{mk}$ such that $B$ lacks any $k$-term cyclic arithmetic progressions mod $mk$. For ease of notation, we set $F := \mathbb{Z}_{mk} \setminus B$. 

Using methods similar to those of Berglund \cite{Berglund}, we have the following result:

\begin{lemma}
\label{lem:2.1}
    If $|D(mk,k)| > 1$, then $b(mk,k) < mk-m$.
\end{lemma}

\begin{proof}
 Suppose $|D(mk,k)| > 1$ and let $B \subseteq \mathbb{Z}_{mk}$ such that $B$ contains no $k$-term cyclic arithmetic progression mod $mk$. We can suppose that $|B| = b(mk,k)$. Then, by Proposition \ref{prop:1.1}, $|B| = b(mk,k) \leq mk-k$. 
 
 Suppose towards contradiction that $b(mk,k) = mk-m$. Since any translate of $B$ will also avoid all $k$-term cyclic arithmetic progressions mod $mk$, without loss of generality, suppose that $0 \not \in B$. Let $\mathbb{Z}_{mk} \setminus B = F = \{f_{1}, \ldots, f_{m}\}$, where $0 = f_{1} < f_{2} < \ldots < f_{m} \leq mk-1$. 
 Note that, for $1 \leq j \leq m-1$, we must have $f_{j+1} - f_{j} \leq k$, since otherwise $(f_{j}+1, f_{j}+2,\ldots,f_{j}+k)$ would be a $k$-term cyclic arithmetic progression mod $mk$ contained in $B$. However, if $f_{j+1}-f_{j} < k$ for some $j$, $1 \leq j \leq m-1$, then $f_{m} = \sum_{j=1}^{m-1} (f_{j+1}-f_{j}) < k(m-1) = mk-k$, meaning $(f_{m}+1, \ldots, f_{m}+k)$ would be a $k$-term cyclic arithmetic progression mod $mk$ with common difference $1$ contained in $B$. So, we must have $f_{j+1}-f_{j} = k$ for $1 \leq j \leq m-1$. Thus, $F = \{0,k,2k,\ldots,(m-1)k\}$.

 Now, since $|D(mk,k)| > 1$ and $1 \in D(mk,k)$, Lemma \ref{prop:2.1} says there exists $d \in D(mk,k)$ such that $d > 1$ and $d \vert k$. We have that $A:=(1,1+d,1+2d,\ldots,1+(k-1)d)$is a $k$-term cyclic arithmetic progression mod $mk$. But no element of $A$ equals $jk \mod mk$ for any integer $0 \leq j < m$. Hence, $F$ does not contain any elements of $A$. So, $A$ is a $k$-term cyclic arithmetic progression mod $mk$ contained in $B$, a contradiction. Hence, we have $b(mk,k) < mk-k$. \qedhere
 \end{proof}

 \begin{corollary}
    Suppose $m,k \in \mathbb{Z}^{+}$, $k \geq 3$, and $m < k$. Then $b(mk,k) = mk-m$ if and only of each divisor of $k$ greater than $1$ is also greater than $m$.
\end{corollary}

\begin{proof}
    The conclusion follows from Lemmas \ref{prop:2.1} and \ref{lem:2.1}, and Proposition \ref{prop:1.1}.
\end{proof}

\section{Bounds for $b(mk,k)$}\label{sec:3}

In this section, we generalize Berglund's findings to bound $b(mk,k)$ for all $m \in \mathbb{Z}^+$. Fix $m,k \in \mathbb{Z}^+$ with $k \geq 3$. In order to prove Theorem \ref{maintheorem}, we will construct $F \subseteq \mathbb{Z}_{mk}$ such that $B = \mathbb{Z}_{mk} \setminus F$ lacks at least one element from every possible $k$-term cyclic arithmetic progression mod $mk$. Thus, $|B|$ will become a lower bound for $b(mk,k)$. 

\subsection{Important Lemmas}
For $d \in \mathbb{Z}^{+}$ and $x \in \mathbb{Z}_d$, let $[x]_d := \{n \in \mathbb{Z} : n \equiv x \mod d\}$. The following lemmas are crucial in the proof of Theorem \ref{maintheorem}:
    
\begin{lemma}
\label{lem:3.1real}
    Let $d,N \in \mathbb{Z}^+$ such that $d \mid N$, and suppose $A$ is a cyclic arithmetic progression mod $N$ with common difference $d$. Then all residue classes of the elements of $A$ lie in a single congruence class of $\mathbb{Z}_{d}$.
\end{lemma}

\begin{proof}
    Let $A \subseteq \mathbb{Z}_{N}$ be a cyclic arithmetic progression with common difference $d$, and let $a \in \{0,1,\ldots,N-1\}$ be a representative of some element in $A$. Fix an arbitrary element $x \in A$. Then we can write $x \equiv a+jd \pmod N$ for some $j \in \mathbb{Z}$. Hence, there exists $t \in \mathbb{Z}$ such that $x = a+jd+tN$. Since $d$ divides $N$, this yields $x \equiv a+jd+tN \equiv a \pmod d$. So, $x \in [a]_d$. Since $x \in A$ was arbitrary, we see that $A \subseteq [a]_d$. That is, all residue classes of the elements of $A$ lie in a single congruence class of $\mathbb{Z}_{d}$.
\end{proof}

\begin{lemma}\label{lem:cons}
Fix $d \in D(mk,k)$, and suppose $A$ is a $k$-term cyclic arithmetic progression mod $mk$ with common difference $d$. Let $\beta \in \{0,1,\ldots,d-1\}$ such that $A \subseteq [\beta]_d$\footnote{The existence of such a $\beta$ is guaranteed by Lemma \ref{lem:3.1real}.}. Set $\alpha := d-\beta$, and define \[S := \{k-\alpha,2k-\alpha,\ldots,mk-\alpha\} \subseteq \mathbb{Z}_{mk}.\] Then $A$ contains $d$ consecutive elements of $S$; that is, there exists an index $j \in \{1,\ldots,m\}$ such that \[\{jk-\alpha,(j+1)k-\alpha,\ldots,(j+d-1)k-\alpha\} \pmod {mk} \subseteq A,\] where the indices $j,j+1,\ldots,j+d-1$ are taken cyclically modulo $m$.
\end{lemma}

\begin{proof}
    First, since $d \in D(mk,k) = \{g \in \mathbb{Z}^+ : g \leq m, g \mid k\}$, we have that $d$ divides $k$. So, $k \equiv 0 \pmod d$. Hence, for every $j \in \{1,\ldots,m\}$, we have \begin{equation*}
        jk -\alpha \equiv jk-(d-\beta) \equiv -d+\beta \equiv \beta \pmod d.
    \end{equation*} So, $S \subseteq [\beta]_d$. On a similar note, since $A \subseteq [\beta]_d$, we can write $A = (a,a+d,a+2d,\ldots,a+(k-1)d) \pmod {mk}$ for some initial term $a \equiv \beta \pmod d$.

    Now, let us show that an arbitrary $x \in \mathbb{Z}_{mk}$ lies in $S$ if and only if $x \equiv -\alpha \pmod k$. First, suppose $x \in S$. Then $x \equiv jk-\alpha \pmod {mk}$ for some $j \in \{1,\ldots,m\}$. Reducing mod $k$ yields $x \equiv -\alpha \pmod k$. Conversely, suppose $x \equiv -\alpha \pmod k$. Then there exists an integer $q$ such that $x+\alpha = qk$. Take $j \in \{1,\ldots,m\}$ such that $j \equiv q \pmod m$ (if $q \equiv 0 \pmod m$, then take $j=m$). Writing $q = j+lm$ for some integer $l$, we have \begin{equation*}
        x+\alpha = qk = (j+lm)k = jk+l(mk).
\end{equation*} So, $x \equiv jk-\alpha \pmod {mk}$. Hence, $x \in S$. So, we see that $x \in S$ if and only if $x \equiv -\alpha \pmod k$.

Next, consider the following equation:\begin{equation}\label{equation1}
    dr \equiv -(\alpha+a) \pmod k.
\end{equation} Since $d$ divides $k$, we can write $k=td$ for some $t \in \mathbb{Z}^+$. So, we can rewrite (\ref{equation1}) as \begin{equation}\label{equation2}
    dr \equiv -(\alpha + a) \pmod {td}.
\end{equation} Since $a \equiv \beta \pmod d$ and $\alpha = d-\beta$, we have that \begin{equation*}
    \alpha+a \equiv \beta +(d-\beta) \equiv d \equiv 0 \pmod d.
\end{equation*} Hence, (\ref{equation2}) is solvable in terms of $r$. 

Now, suppose $r_0$ is one such solution to (\ref{equation2}). Then all solutions modulo $td$ are $r_0,r_0+t,r_0+2t,\ldots,r_0+(d-1)t$. Indeed, for each $j \in \{0,1,\ldots,d-1\}$, \begin{equation*}
    d(r+jt) \equiv dr+j(dt) \equiv dr \equiv -(\alpha+a) \pmod {td}.
\end{equation*} Moreover, suppose $r_0+lt \equiv r_0+l't \pmod {td}$ for some integers $0 \leq l,l' \leq d-1$. Then \begin{equation*}
    (l-l')t \equiv 0 \pmod {td} \implies l-l' \equiv 0 \pmod d.
\end{equation*} Since $|l-l'| < d$, it follows that $l-l' = 0$, i.e. $l=l'$. Hence, the solutions $r_0,r_0+t,r_0+2t,\ldots,r_0+(d-1)t$ are distinct modulo $td$. That is, there are $d$ distinct solutions to (\ref{equation2}) modulo $td$.

Finally, for such a solution $r_0$, consider the corresponding elements of $A$: \begin{equation*}
    A' :=\{a+dr_0,a+d(r_0+t),a+d(r_0+2t),\ldots,a+d(r_0+(d-1)t)\} \pmod {mk}.
\end{equation*} By construction, the first element of $A'$ satisfies \begin{equation*}
    a+dr_0 \equiv a-(\alpha+a) \equiv -\alpha \pmod k.
\end{equation*} So, $a+dr_0 \in S$. Hence, we can write $a+dr_0 \equiv jk-\alpha \pmod {mk}$ for some $j \in \{1,\ldots,m\}$. Each subsequent element in $A'$ differs from the previous one by $dt = k$. So, reducing $A'$ modulo $mk$ yields \begin{equation*}
    jk-\alpha, (j+1)k-\alpha,\ldots,(j+d-1)k-\alpha \pmod {mk},
\end{equation*} with indices taken cyclically modulo $m$. Hence, there exists an index $j \in \{1,\ldots,m\}$ such that \begin{equation*}
    \{jk-\alpha,(j+1)k-\alpha,\ldots,(j+d-1)k-\alpha\} \pmod {mk} \subseteq A,
\end{equation*} where the indices $j,j+1,\ldots,j+d-1$ are taken cyclically modulo $m$.
\end{proof}

\begin{lemma}
\label{lem:nonemptycon}
    Let $d,N \in \mathbb{Z}^+$ such that $d \leq N$. For each $i \in \mathbb{Z}_{N}$, define $X_{i} := \{i,i+1,\ldots,i+d-1\} \pmod N$. Let $C = \{d-1,d,d+1,\ldots,N-1\}$. Then $X_i \cap C \neq \emptyset$ for all $i \in \mathbb{Z}_N$.
\end{lemma}

\begin{proof}
    Suppose towards contradiction there exists $i \in \mathbb{Z}_N$ such that $X_i \cap C = \emptyset$. Then $X_i \subseteq \mathbb{Z}_{N} \setminus C$. Hence, $|X_i| \leq |\mathbb{Z}_{N} \setminus C|$. But $|X_i| = |\{i,i+1,\ldots,i+d-1\}| =d$ and $|\mathbb{Z}_N \setminus C| = |\{0,1,2,\ldots,d-2\}| = d-1$. So, $d \leq d-1$, a contradiction. Hence, $X_i \cap C \neq \emptyset$ for all $i \in \mathbb{Z}_N$.
\end{proof}

\subsection{Construction of $F$}

Let us now construct $F \subseteq \mathbb{Z}_{mk}$ as follows:

\begin{construction}
\label{cons:3.3}
Let $D(mk,k) = \{d_0,d_1,\ldots,d_{j}\}$, $d_0 < d_1 < \cdots < d_{j}$, for some $j \geq 0$, and take $d_{-1} := 0$. For each $i \in \{0,\ldots,j\}$, define \begin{equation*}
    F_{i} := \bigcup_{\alpha = d_{i-1}+1}^{d_i}\{d_{i}k-\alpha,(d_{i}+1)k-\alpha,\ldots,mk-\alpha\}.
\end{equation*} Then take $F := \bigcup_{i=0}^{j}F_{i}$.
    
\end{construction}

\begin{example}
    Consider $D(3k,k) = \{1,3\}$ (and suppose $k=15$). We have \begin{align*}
        F_{0} &= \bigcup_{\alpha=1}^{1}\{k-\alpha,2k-\alpha,3k-\alpha\} = \{k-1,2k-1,3k-1\} \\
        F_{1} &= \bigcup_{\alpha=2}^{3}\{3k-\alpha\} = \{3k-2,3k-3\}.
    \end{align*} So, \begin{align*}
        F = F_{0} \cup F_{1} = \{k-1,2k-1,3k-3,3k-2,3k-1\}.
    \end{align*} Plugging in $k = 15$ yields $F = \{14,29,42,43,44\}$.
\end{example}

\begin{example}
    Consider $D(10k,k) = \{1,3,9\}$ (and suppose $k=9$). We have \begin{align*}
        F_{0} &= \bigcup_{\alpha=1}^{1}\{k-\alpha,2k-\alpha,\ldots,10k-\alpha\} = \{k-1,2k-1,\ldots,10k-1\} \\
        F_{1} &= \bigcup_{\alpha=2}^{3}\{3k-\alpha,4k-\alpha,\ldots,10k-\alpha\} = \{3k-2,\ldots,10k-2\}\cup \{3k-3,\ldots,10k-3\} \\
        F_{2} &= \bigcup_{\alpha=4}^{9}\{9k-\alpha,10k-\alpha\} = \{9k-4,10k-4\} \cup \cdots \cup \{9k-9,10k-9\}.
    \end{align*} So, $F = F_{0} \cup F_{1} \cup F_{2}$. Plugging in $k=9$ yields $F = \{8,17,24,25,26,33,34,35,42,43,44,51,52,$ \\ $53,60,61,62,69,70,71,72,73,74,75,76,77,78, 79,80,81,82,83,84,85,86,87,88,89\}$.
\end{example}

\begin{lemma}\label{lem:disjointf}
    Using the same notation as Construction \ref{cons:3.3}, we have that $F_{i_1} \cap F_{i_2} = \emptyset$ for all $i_{1},i_{2} \in \{0,1,\ldots,j\}$ such that $i_1 \neq i_2$.
\end{lemma}

\begin{proof}
    Fix $i_{1},i_{2} \in \{0,1,\ldots,j\}$, and without loss of generality, assume $i_1 < i_2$. Then \begin{align*}
        F_{i_1} &= \bigcup_{\alpha = d_{i_1-1}+1}^{d_{i_1}}\{d_{i_1}k-\alpha,(d_{i_1}+1)k-\alpha,\ldots,mk-\alpha\} \qquad \text{and} \\
        F_{i_2} &= \bigcup_{\alpha = d_{i_2-1}+1}^{d_{i_2}}\{d_{i_2}k-\alpha,(d_{i_2}+1)k-\alpha,\ldots,mk-\alpha\},
    \end{align*} where $d_{i_{1}-1},d_{i_1},d_{i_{2}-1},d_{i_2} \in D(mk,k)$.
    
    Let $x_1 \in F_{i_1}$ and $x_{2} \in F_{i_2}$. Then there exists $\alpha_1 \in \{d_{{i_1}-1}+1,\ldots,d_{i_1}\}$, $\alpha_2 \in \{d_{{i_2}-1}+1,\ldots,d_{i_{2}}\}$, $t_1 \in \{0,\ldots,m-d_{i_1}\}$, and $t_2 \in \{0,\ldots,m-d_{i_2}\}$ such that \begin{equation*}
        x = (d_{i_1}+t_1)k-\alpha_1 \quad \text{and} \quad y=(d_{i_2}+t_2)k-\alpha_2.
    \end{equation*} Suppose towards contradiction that $x_1 = x_2$. Then \begin{equation*}
        (d_{i_1}+t_1)k-\alpha_1 = (d_{i_2}+t_2)k-\alpha_2,
    \end{equation*} implying that \begin{equation}\label{equation3}
        \Big((d_{i_2}+t_2)-(d_{i_1}+t_1)\Big)k = \alpha_2-\alpha_1.
    \end{equation} Now, the left-hand side of (\ref{equation3}) is divisible by $k$. On the other hand, since $i_1 < i_2$, we have that $\alpha_1 \leq d_{i_1} < d_{i_2-1}+1 \leq \alpha_2$, so that $\alpha_2 - \alpha_1 \geq (d_{{i_2}-1}+1)-d_{i_1} \geq 1$. Similarly, we have that $\alpha_2 \leq d_{i_2}$ and $\alpha_1 \geq d_{i_{1}-1} +1$, so that $\alpha_2 - \alpha_1 \leq d_{i_2}-(d_{i_{1}-1}+1) = d_{i_2}-d_{i_{1}-1}-1 \leq k-1$, since $d_{i_1},d_{i_2} \in D(mk,k) \subseteq \{1,2,\ldots,k\}$. Hence, \begin{equation*}
        1 \leq \alpha_2-\alpha_1 \leq k-1,
    \end{equation*} so the right-hand side of (\ref{equation3}) is not divisible by $k$, a contradiction. Hence, $x_1 \neq x_2$. Since $x_1 \in F_{i_1}$ and $x_2 \in F_{i_2}$ were arbitrary, we see that $F_{i_1} \cap F_{i_{2}} = \emptyset$.
\end{proof}

Now, we have the following:

\begin{proposition}
\label{prop:sizeoff}
    Construct $F \subseteq \mathbb{Z}_{mk}$ via Construction \ref{cons:3.3}. Using the same notation as above, we have \begin{align}
        |F| = \sum_{i=0}^{j}(d_i-d_{i-1})(m-d_i+1).
    \end{align}
\end{proposition}

\begin{proof}
    Fix $d_i \in D(mk,k)$ for some $i \in \{0,1,\ldots,j\}$. Clearly, \[\{d_ik-\alpha_1,(d_i+1)k-\alpha_1,\ldots,mk-\alpha_1\} \cap \{d_ik-\alpha_2,(d_i+1)k-\alpha_2,\ldots,mk-\alpha_2\} = \emptyset\] for any $d_{i-1} < \alpha_1,\alpha_2 \leq d_i$ such that $\alpha_1 \neq \alpha_2$. So, \begin{align*}
        |F_i| &= \Big|\bigcup_{\alpha=d_{i-1}+1}^{d_i}\{d_ik-\alpha,(d_i+1)k-\alpha,\ldots,mk-\alpha\}\Big| \\
        &= \sum_{\alpha=d_{i-1}+1}^{d_i}|\{d_{i}k-\alpha,(d_i+1)k-\alpha,\ldots,mk-\alpha\}| \\
        &= \sum_{\alpha=d_{i-1}+1}^{d_i}(m-d_i+1) = (d_{i}-d_{i-1})(m-d_{i}+1).
    \end{align*} Finally, $F_{i_1} \cap F_{i_2} = \emptyset$ for all $i_1 \neq i_2$ by Lemma \ref{lem:disjointf}. Hence, $|F| = \sum_{i=0}^{j}(d_i-d_{i-1})(m-d_i+1)$.
\end{proof}

\subsection{Proof of Theorem \ref{maintheorem}}

We are now able to prove Theorem \ref{maintheorem}. For ease of notation, let $[x,y] := \{x,x+1,\ldots,y\}$ for integers $x,y$ such that $x<y$, and let $[x,y] \pmod N := \{x,x+1,\ldots,y\} \pmod N$ for integers $x,y$ and $N \in \mathbb{Z}^+$.

\begin{proof}[Proof of Theorem \ref{maintheorem}] First, Proposition \ref{prop:1.1} yields $mk - m \geq b(mk,k)$. On the other hand, let $D(mk,k) = \{d_0,d_1,\ldots,d_{j}\}$, $d_0 < d_1 < \cdots < d_{j}$, for some $j \geq 0$, and let $d_{-1} := 0$. We will show that if we construct $F$ via Construction \ref{cons:3.3}, then $B := \mathbb{Z}_{mk} \setminus F$ will not contain any $k$-term cyclic arithmetic progressions mod $mk$. Then, the definition of $b(mk,k)$ and Proposition \ref{prop:sizeoff} will yield \begin{equation*}
    b(mk,k) \geq |B| = |\mathbb{Z}_{mk} \setminus F| = mk-\sum_{i=0}^{j}(d_i-d_{i-1})(m-d_{i}+1).
\end{equation*}

Fix $i \in [0,j]$, and suppose $A$ is a $k$-term cyclic arithmetic progression mod $mk$ with common difference $d_i$. Modulus $d_i$ admits exactly $d_i$ congruence classes, namely $[0]_{d_i},[1]_{d_{i}},\ldots,[d_{i}-1]_{d_{i}}$. By Lemma \ref{lem:3.1real}, there exists a unique $\beta \in [0,d_i-1]$ such that $A \subseteq [\beta]_{d_{i}}$. Consider this $\beta$.

By Lemma \ref{lem:cons}, $A$ contains exactly $d_{i}$ consecutive elements of $S := \{k-\alpha_\beta,2k-\alpha_\beta,\ldots,mk-\alpha_\beta\} \subseteq \mathbb{Z}_{mk}$, where $\alpha_\beta := d_{i} - \beta$. Let $S_A \subseteq S$ denote this consecutive subset. That is, there exists an index $j \in \{1,\ldots,m\}$ such that \begin{equation*}
    S_A := \{jk-\alpha_\beta,(j+1)k-\alpha_\beta,\ldots,(j+d_i-1)k-\alpha_\beta\} \pmod {mk}.
\end{equation*}

Now, there is a clear bijection between $\mathbb{Z}_m$ and $S$, namely \begin{equation*}
    \Phi : \mathbb{Z}_m \to S, \quad \Phi(j) = (j+1)k-\alpha_\beta.
\end{equation*} So, $S_{A} = \Phi([j-1,j+d_i-2] \pmod m)$. Put in another way, $S_A = \Phi([l,l+d_i-1] \pmod m)$ for $l=j-1$.

Let $C := \Phi([d_i-1,m-1]) = \{d_ik-\alpha_\beta,(d_{i}+1)k-\alpha_\beta,\ldots,mk-\alpha_\beta\}$. Then Lemma \ref{lem:nonemptycon} yields that $C \cap S_A \neq \emptyset$. That is, $C$ contains at least one term of $A$ (specifically, $C$ contains at least one element of $A$ no matter what $j \in [1,m]$ is). So, if $C \subseteq F$, then $B$ cannot contain all elements of $A$ and hence does not contain $A$. 

Now, $\alpha_\beta = d_i-\beta \in [1,d_i]$, since $\beta \in [0,d_i-1]$. Since $[1,d_i] = \bigcup_{q=0}^{j}[d_{q-1}+1,d_q]$, there exists a unique $q \in [0,i]$ such that $\alpha_\beta \in [d_{q-1}+1,d_q]$. And since $d_0 < d_1 < \cdots < d_{j}$, we have that $d_q \leq d_i$. So,
\begin{align*}
    C  \subseteq \bigcup_{\alpha=d_{q-1}+1}^{d_{q}}\{d_qk-\alpha,(d_q+1)k-\alpha,\ldots,mk-\alpha\} = F_{q} \subseteq F.
\end{align*}

Hence, $B$ does not contain all elements of $A$, so $B$ does not contain $A$. Since $\beta \in [0,d_i-1]$ was arbitrary, a similar argument applied to any $\beta \in [0,d_i-1]$ shows that $B$ does not contain any progressions contained in any congruence class of $d_i$. So, $B$ does not contain any $k$-term cyclic arithmetic progressions mod $mk$ with common difference $d_i$. Finally, since $i \in [0,j]$ was arbitrary, a similar argument applied to any $i \in [0,j]$ shows that $B$ does not contain any $k$-term cyclic arithmetic progressions mod $mk$ with any common difference in $D(mk,k)$. So, by Lemma \ref{lem:1.2}, we have that $B$ does not contain any $k$-term arithmetic progressions mod $mk$. Hence, $b(mk,k) \geq |B|$ and our theorem is proven.
\end{proof}

\subsection{Examples \& Applications}

\begin{example}
    Let $m=3$ and $k=4$. Then $D(mk,k) = D(12,4) = \{1,2\}$ and Theorem \ref{maintheorem} yields $9 \geq b(12,4) \geq 7$.
\end{example}

\begin{example}
    Let $m=k=9$. Then $D(mk,k) = D(81,9) = \{1,3,9\}$ and Theorem \ref{maintheorem} yields $72 \geq b(81,9) \geq 52$.
\end{example}

As immediate consequences of Theorem \ref{maintheorem}, we obtain alternative proofs of two results from \cite{Berglund}:

\begin{proposition*}[Partial Restatement]
    If $D(mk,k) = \{1\}$, then $b(mk,k) = mk-m$.
\end{proposition*}

\begin{proof}
    Let $D(mk,k) = \{1\}$. Then Theorem \ref{maintheorem} yields \begin{align*}
        mk-m \geq b(mk,k) \geq mk-(1-0)(m-1+1) = mk-m,
    \end{align*} implying that $b(mk,k) = mk-m$.
\end{proof}

\begin{remark}
Note that the proof of Theorem \ref{maintheorem} relies only on the first part of Proposition \ref{prop:1.1}, specifically the inequality $b(mk,k) \leq mk-m$. Hence, we are justified to use Theorem \ref{maintheorem} to establish the second half of Proposition \ref{prop:1.1}.
\end{remark}

\begin{proposition}
    If $p \geq 3$ is prime, then $(p-1)^2 \leq b(p^2,p) \leq p(p-1)$.
\end{proposition}

\begin{proof}
   Let $p \geq 3$ be prime. Observe that $b(p^2,p)$ corresponds to the case of $b(mk,k)$ in which $m=k=p$ is prime. So, $D(mk,k) = D(p^2,p) = \{g \in \mathbb{Z}^+ : g \leq p^2, g \mid p\} = \{1,p\}$. That is, $|D(p^2,p)| > 1$, since $p \geq 3$. So, Lemma \ref{lem:2.1} yields $b(p^2,p) = b(mk,k) \leq mk-m = p^2-p = p(p-1)$.

   On the other hand, letting $d_{-1} := 0, d_{1} := 1,$ and $d_2 := p$, Theorem \ref{maintheorem} yields \begin{align*}
       b(p^2,p) &\geq p^2-\sum_{i=0}^{2}(d_i-d_{i-1})(p-d_i+1) \\
       &= p^2-\Big((1-0)(p-1+1)+(p-1)(p-p+1) \Big ) \\
       &= p^2-\Big(p+(p-1)\Big) \\
       &= p^2-2p+1 \\
       &= (p-1)^2.
   \end{align*} All together, we see that $(p-1)^2 \leq b(p^2,p) \leq p(p-1)$.
\end{proof}

\section{Implications of our Bounds}\label{sec:implications}
Recall from Section \ref{sec:1.2} that $b(mk,k)$ is the independence number of the cyclic Van der Waerden hypergraph $H_{mk,k}$, where $m,k \in \mathbb{Z}^+$ with $k \geq 3$.

In this section, we use our constructions of $B,F \subseteq \mathbb{Z}_{mk}$ to tightly bound chromatic number $\chi(mk,k)$, thereby providing a lower bound of the cyclic Van der Waerden number $W_c(k,r)$ for various $r \geq 2$.

\subsection{Chromatic Number $\chi(mk,k)$}
For integers $m,k \in \mathbb{Z}^+$ with $k \geq 3$, recall that the \textit{chromatic number} $\chi(mk,k)$ of $\mathbb{Z}_{mk}$ is the minimum number of subsets necessary to partition $\mathbb{Z}_{mk}$ such that no subset contains any $k$-term cyclic arithmetic progressions mod $mk$. In this section, we tightly bound $\chi(mk,k)$ in various cases based on the relationship between $m$ and $k$.

First, for $N \in \mathbb{Z}^+$, let us define the \textit{cyclic difference} between two elements $\alpha,\beta \in \mathbb{Z}_{N}$, $\alpha < \beta$, to be $(\alpha-\beta) \pmod N$. For example, the cyclic difference between $2$ and $7$ in $\mathbb{Z}_8$ is $3$.

Also, for an arbitrary progression $(a_0,a_1,\ldots,a_j) \subseteq \mathbb{Z}_N$, let $\sigma(a_0,a_1,\ldots,a_j) = (a_{\sigma(0)},a_{\sigma(1)},\ldots,a_{\sigma(j)})$ be the permutation such that $a_{\sigma(0)} < a_{\sigma(1)} < \cdots < a_{\sigma(j)}$, with the order $<$ being taken in $\mathbb{Z}$. We say that the progression $(a_0,a_1,\ldots,a_j)$ \textit{cycles} if $\sigma$ is not the identity permutation, i.e. there is some $i$ such that $\sigma(i) \neq i$.

\begin{example}
    In $\mathbb{Z}_{8}$, $\sigma(1,3,5) = (1,3,5)$ and $\sigma(5,7,1) = (1,5,7)$. So, $(1,3,5)$ does not cycle while $(5,7,1)$ does.
\end{example}

We are now ready for our first proposition:

\begin{proposition}\label{prop:4.1new}
    Let $m \in \mathbb{Z}^+$. If $k > m$, then $\chi(mk,k) = 2$.
\end{proposition}

\begin{proof}
    Suppose $k > m$, and construct $F \subseteq \mathbb{Z}_{mk}$ via Construction \ref{cons:3.3}. Let $B := \mathbb{Z}_{mk} \setminus F$. From the proof of Theorem \ref{maintheorem}, we know that $B$ does not contain any $k$-term cyclic arithmetic progressions mod $mk$. Our goal is to show that $F$ does not either. If $k > |F|$, then we are clearly done, so we may assume that $k \leq F$.  

    First, let us consider possible progressions that do not cycle. We start with
    $F_0 := (k-1,2k-1,\ldots,mk-1) \subseteq F$, since $d_0 :=1 \in D(mk,k)$ always holds. $F_0$ is an $m$-term arithmetic progression mod $mk$, but since $m < k$, this is not a progression we must worry about. Now, suppose $D(mk,k) = \{1,2,\ldots,m\}$. Then \begin{align*}
     (\min(F_0),\min(F_1),\ldots,\min(F_{m-1})) &= (k-1,2k-2,\ldots,mk-m) \qquad \ \ \  \text{and} \\
        (\max(F_0),\max(F_1),\ldots,\max(F_{m-1})) &= (mk-1,mk-2,\ldots,mk-m)
    \end{align*} are two $m$-term arithmetic progressions mod $mk$ contained in $F$. But again, since $m<k$, we need not worry about these progressions. And finally, any other progressions that do not cycle in $F$ are shorter than these three by construction (no matter what $D(mk,k)$ actually is), so can conclude that $F$ does not contain any such progressions that do not cycle.

    Now, let us consider possible progressions that do cycle. Lemmas \ref{lem:1.2} and \ref{prop:2.1} tell us that we must consider any such progression with common difference at most $m$, by definition of $D(mk,k)$. Now, $\min(F) = k-1$ and $\max(F) = mk-1$, since $1 \in D(mk,k)$ always holds. The cyclic difference between these two elements is $(k-1)-(mk-1) \pmod {mk} = k$. So, any $k$-term progression that does cycle has common difference at least $k$. But since $m < k$, we need not worry about such progression either. Hence, $F$ also does not contain any $k$-term cyclic arithmetic progressions mod $mk$ that do cycle.

    And so, we can partition $\mathbb{Z}_{mk}$ into subsets $B$ and $F$ such that neither subset contains any $k$-term cyclic arithmetic progression mod $mk$. So, $\chi(mk,k) = 2$.
\end{proof}

As seen in the above proof, Construction \ref{cons:3.3} yields $F_0 := (k-1,2k-1,\ldots,mk-1) \subseteq F$, which is a $k$-term cyclic arithmetic progression mod $mk$ with common difference $k$. So, if $k \leq m$, then $k \in D(mk,k)$ and $F$ contains such a $k$-term progression that we want to avoid. So, we cannot simply partition $\mathbb{Z}_{mk}$ into subsets $B$ and $F$ like the case where $k > m$. Instead, we must adjust our construction slightly.

Let us first consider the case where $m=k$. Below is an important lemma which we will use to bound $\chi(mk,k) = \chi(k^2,k)$:

\begin{lemma}\label{lem:partF}
    Let $m = k \geq 3$, and construct $F$ via Construction \ref{cons:3.3}. Arrange the elements of $F$ in increasing order (viewed in $\mathbb{Z}$). Partition $F$ into two subsets $F'$ and $F''$ as follows: place the $\lfloor\frac{k}{2}\rfloor$ smallest elements of $F$ into $F'$, the next $\lfloor\frac{k}{2}\rfloor$ elements into $F''$, and then continue alternating in this way until $F = F' \cup F''$. Then, neither $F'$ nor $F''$ contain all elements of $F_0 := \{k-1,2k-1,\ldots,k^2-1\}$.
\end{lemma}

\begin{remark}
    For convenience, let us refer to a group of consecutive elements in $F$ as a \textit{consecutive segment}, and let the \textit{length} of the segment be the number of elements it contains. Hence, we are alternating consecutive segments of length (at most) $\lfloor\frac{k}{2}\rfloor$ between $F'$ and $F''$.
\end{remark}

\begin{proof}
    First, suppose $k$ is prime. Then $D(k^2,k) = \{1,k\}$. So, the $k$ smallest elements of $F$ are $k-1,2k-1,\ldots,(k-1)k-1$, and $k^2-k$. Note that $\lceil\frac{k}{2}\rceil k - 1 \leq k^2-k$, so $\lceil\frac{k}{2}\rceil k-1$ is one of the $k$ smallest elements of $F$. Hence, constructing $F'$ and $F''$ as above yields $k-1 \in F'$ and $\lceil\frac{k}{2}\rceil k-1 \in F''$. So, neither $F'$ nor $F''$ contain all elements of $F_{0}$.

    Next, suppose $k$ is not prime. Denote $D(k^2,k) = \{d_0,d_1,\ldots,d_{j}\}$, where $d_{j} = k$ and $d_{j-1} \leq \frac{k}{2}$. Note that $d_{j-1} \leq \lfloor\frac{k}{2}\rfloor < 2d_{j-1}$, and $d_{j-1} = \frac{k}{2}$ if and only if $k$ is even. So, let us first consider when $k$ is even. Without loss of generality, assume $(\frac{k}{2})k-1 \in F'$. When $F$ is arranged in increasing order (viewed in $\mathbb{Z}$), there are exactly $d_{j-1}-1 = \frac{k}{2}-1$ elements strictly between $(\frac{k}{2})k-1$ and $(\frac{k}{2}+1)k-1$. Hence, $(\frac{k}{2}+1)k-1 \in F''$. So, when $k$ is even, we have that neither $F'$ nor $F''$ contains all elements of $F_0$.

    Finally, let us consider the case when $k$ is odd and not prime, i.e. $1<d_{j-1} < \lfloor\frac{k}{2}\rfloor$. Let \begin{equation*}
        x_1 := \lfloor\frac{k}{2}\rfloor k-1,  \quad x_2 := \lfloor\frac{k}{2}\rfloor(k+1)-1, \quad x_3 := \lfloor\frac{k}{2}\rfloor(k+2)-1.
    \end{equation*} Clearly, $x_1,x_2,x_3 \in F_0$. Without loss of generality, assume $x_2 \in F'$. If $x_1 \in F''$, then we are done. So, assume $x_1 \in F'$ as well. When $F$ is arranged in increasing order (viewed in $\mathbb{Z}$), there are exactly $d_{j-1}+1$ elements between $x_1$ and $x_2$ (inclusive), and there are exactly $d_{j-1}+1$ elements between $x_{2}$ and $x_3$ (inclusive). Hence, the consecutive segment of elements of $F$ containing $x_1$ and $x_2$ has the same length as the consecutive segment of $F$ containing $x_2$ and $x_3$. So, there cannot be a consecutive segment of $F$ strictly between $x_2$ and $x_3$ contained entirely in $F''$ (which would then allow $x_3 \in F'$). On the other hand, since $d_{j-1}+1 \leq \lfloor\frac{k}{2}\rfloor < 2d_{j-1}+1$\footnote{To see this upper bound, write $d_{j-1} = \frac{k}{p}$, where $p$ is the smallest prime factor of $k$. Clearly, $p \geq 3$. So, $2d_{j-1}+1 = 2\cdot \frac{k}{p}+1 \geq \frac{2k}{3}+1$. Since $\lfloor\frac{k}{2}\rfloor = \frac{k-1}{2}$, it suffices to show that $\frac{k-1}{2} < \frac{2k}{3} + 1$. Multiplying both sides by $6$ yields $3(k-1) < 4k+6 \implies -3 < k+6$, which is clearly true for all $k \geq 3$. }, the consecutive segment of length $\lfloor\frac{k}{2}\rfloor$ containing $x_1$ and $x_2$ cannot also contain $x_3$. Hence, we must have $x_3 \in F''$. And so, neither $F'$ nor $F''$ contain all elements of $F_0$.
    
    All cases considered, neither $F'$ nor $F''$ contain all elements of $F_0$.
\end{proof}

We can now prove the following:

\begin{proposition}\label{prop:4.2new}
    Let $k \geq 3$. Then $2 \leq \chi(k^2,k) \leq 3$.
\end{proposition}

\begin{proof}
    Clearly, $\chi(k^2,k) \geq 2$. On the other hand, construct $F \subseteq \mathbb{Z}_{k^2}$ via Construction \ref{cons:3.3} (with $m=k$), and let $B := \mathbb{Z}_{k^2} \setminus F$. From the proof of Theorem \ref{maintheorem}, we know that $B$ does not contain any $k$-term cyclic arithmetic progressions mod $k^2$. Decompose $F$ into $F'$ and $F''$ via the method described in Lemma \ref{lem:partF}. Let us show that neither $F'$ nor $F''$ contain any $k$-term cyclic arithmetic progressions mod $k^2$. 

    First, let us consider possible progressions in $F$ that do not cycle. We start with $F_0 := (k-1,2k-1,\ldots,k^2-1) \subseteq F$, since $1 \in D(k^2,k)$ always holds. By Lemma \ref{lem:partF}, neither $F'$ nor $F''$ contains all of $F_0$. On the other hand, since $k \geq 3$, $\{1,2,\ldots,k\} \not\subseteq D(k^2,k)$. Hence, the construction of $F$ yields that the only other possible $k$-term arithmetic progression mod $k^2$ contained in $F$ is $(k^2-k,\ldots,k^2-2,k^2-1)$. But $k^2-k,\ldots,k^2-2,k^2-1$ are the $k$ largest elements of $F$. Hence, the constructions of $F'$ and $F''$ yield that neither can contain all elements of $(k^2-k,\ldots,k^2-2,k^2-1)$, since we alternate consecutive segments of $F$ of length (at most) $\lfloor\frac{k}{2}\rfloor < k$ between $F'$ and $F''$. So, we can conclude that neither $F'$ nor $F''$ contain any $k$-term arithmetic progressions mod $k^2$ that do not cycle.

    Next, let us consider possible progressions that do cycle. Lemmas \ref{lem:1.2} and \ref{prop:2.1} tell us that we must consider any such progression with common difference at most $k$, by definition of $D(k^2,k)$. Now, $\min(F) = k-1$ and $\max(F) = k^2-1$, since $1 \in D(k^2,k)$ always holds. The cyclic difference between these two elements is $(k-1)-(k^2-1) \pmod {k^2} = k$. So, any $k$-term progression that does cycle has common difference at least $k$, which is the largest common difference for us to consider. However, since $(1,2,\ldots,k) \not\subseteq D(mk,k)$, the only $k$-term cyclic arithmetic progression in $F$ with common difference $k$ is $F_0$, which we know is not contained in either $F'$ nor $F''$. Hence, neither $F'$ nor $F''$ contain any $k$-term cyclic arithmetic progressions mod $k^2$ either.

    And so, we can partition $\mathbb{Z}_{mk}$ into subsets $B,F',F''$ such that neither $B,F'$, nor $F''$ contain any $k$-term cyclic arithmetic progressions mod $k^2$. Hence, $\chi(k^2,k) \leq 3$. All together, we have that $2 \leq \chi(k^2,k) \leq 3$.
\end{proof}

Finally, let us consider the case where $k < m$. We first need the following lemma:

\begin{lemma}\label{lemforklessm}
    Let $k \geq 3$ and $S \subseteq \{i \in \mathbb{Z} : 0 \leq i \leq k^2-1\}$ be arbitrary. If $S$ does not contain any $k$-term cyclic arithmetic progression mod $k^2$, then $S$ does not contain any $k$-term cyclic arithmetic progression mod $mk$ for any $m \geq k$.
\end{lemma}

\begin{proof}
    For any $m \geq k \geq 3$, $D(mk,k) = \{1 \leq g \leq m : g \vert k \} = \{1 \leq g \leq k : g | k \} = D(k^2,k)$. So, by Lemmas \ref{lem:1.2} and \ref{prop:2.1}, if $S \subseteq \{i \in \mathbb{Z} : 0 \leq i \leq k^2-1\}$ does not contain any $k$-term cyclic arithmetic progressions mod $k^2$, then $S$ does not contain any $k$-term cyclic arithmetic progressions mod $mk$ either.  
\end{proof}

\begin{proposition}\label{prop:4.3new}
    Let $3 \leq k < m$. Then $2 \leq \chi(mk,k) \leq 3 + \lceil\frac{(m-k)\cdot k}{k-1}\rceil$.
\end{proposition}

\begin{proof}
    Clearly, $\chi(mk,k) \geq 2$. On the other hand, construct $F \subseteq \mathbb{Z}_{mk}$ via construction \ref{cons:3.3}, and let $B := \mathbb{Z}_{mk} \setminus F$. From the proof of Theorem \ref{maintheorem}, we know that $B$ does not contain any $k$-term cyclic arithmetic progressions mod $mk$.

    Let us define $F^k := F \cap \{i \in \mathbb{Z} : 0 \leq i \leq k^2-1\}$. That is, $F^k$ consists of the values of $F$ that occur in the case where $k = m$ (since $D(mk,k) = D(k^2k)$ for $m > k$). Decompose $F^k$ into subsets $F^{k'}$ and $F^{k''}$ in a similar fashion to the construction of $F'$ and $F''$ in Lemma \ref{lem:partF}. Then the proof of Proposition \ref{prop:4.2new} yield that neither $F^{k'}$ nor $F^{k''}$ contain any $k$-term cyclic arithmetic progressions mod $k^2$. Then, since $m > k$, Lemma \ref{lemforklessm} yields that neither $F^{k'}$ nor $F^{k''}$ contain any $k$-term cyclic arithmetic progressions mod $mk$.

    Finally, let $E := F \setminus F^k$. Trivial calculations yield $|E| = (m-k)\cdot k$. So, if we arbitrarily partition $E$ into subsets of size at most $k-1$, call them $E_1,E_2,\ldots,E_{\gamma}$ with $\gamma = \lceil\frac{(m-k)\cdot k}{k-1}\rceil$, then no subset $E_{i}$, $i \in \{1,\ldots,\gamma\}$, contains $k$-term cyclic arithmetic progressions mod $mk$ trivially. Hence, we can partition $\mathbb{Z}_{mk}$ into $B,F^{k'},F^{k''},E_{1}\ldots,E_{\gamma}$ such that no subset contains any $k$-term cyclic arithmetic progressions mod $mk$. So, $\chi(mk,k) \leq 3+\lceil\frac{(m-k)\cdot k}{k-1}\rceil$. All together, we see that $2 \leq \chi(mk,k) \leq 3+\lceil\frac{(m-k)\cdot k}{k-1}\rceil$.
\end{proof}

\subsection{Lower Bounds for Cyclic Van der Waerden Numbers $W_{c}(k,r)$}

Using the bounds of $\chi(mk,k)$ for various $m,k \in \mathbb{Z}^+$, $k \geq 3$, found in the previous section, we are now able to bound $W_{c}(k,r)$ for various $r \geq 2$ from below. Recall that the cyclic Van der Waerden number $W_{c}(k,r)$ denotes the smallest modulus $N$ such that for all positive integers $M \geq N$, if $\mathbb{Z}_{M}$ is $r$-coored, then there exists a monochromatic $k$-term cyclic arithmetic progression mod $M$ in $\mathbb{Z}_{M}$. Also, recall that if $\chi(N,k) = r$, then $W_{c}(k,r) > N$ for $N \in \mathbb{Z}^+$.

In each of the following propositions, let $k \geq 3$.

\begin{proposition}
    $W_c(k,2) > k(k-1)$.
\end{proposition}

\begin{proof}
    By Proposition \ref{prop:4.1new}, $\chi(mk,k) = 2$ for all $m \leq k-1$. So, $W_{c}(k,2) > k(k-1)$.
\end{proof}

\begin{proposition}\label{prop:finalneed}
    $W_c(k,3) > k^2$.
\end{proposition}

\begin{proof}
    By Proposition \ref{prop:4.2new}, $\chi(k^2,k) \leq 3$. Since we can partition $\mathbb{Z}_{k^2}$ into three disjoint subsets, none of which contain any $k$-term cyclic arithmetic progresions mod $k^2$, we have $W_{c}(k,3) > k^2$.
\end{proof}

\begin{proposition}
    For all $k < m$, $W_c(k,3+\lceil\frac{(m-k)\cdot k}{k-1}\rceil) > mk$.
\end{proposition}

\begin{proof}
    The result follows from Proposition \ref{prop:4.3new} with similar reasoning to the proof of Proposition \ref{prop:finalneed}.
\end{proof}

\section{Areas for Future Research}\label{sec:future}

Those interested may wish to consider raising some of the lower bounds in this paper. First and foremost, the set $F \subseteq \mathbb{Z}_{mk}$ is not necessarily of minimal size. To see this, we note that for $d,N \in \mathbb{Z}^+$ such that $d \leq N$, the minimum size of a subset of $\mathbb{Z}_{N}$ intersecting every $d$-element consecutive subset of $\mathbb{Z}_{N}$ is $\lceil\frac{N}{d}\rceil$\footnote{This follows directly from the pigeonhole principle.}, while the set $C$ constructed in Lemma \ref{lem:nonemptycon} is of size $|C| = |\{d-1,d,d+1,\ldots,N-1\}| = N-d+1 \geq \lceil\frac{N}{d}\rceil$ (and the construction of $F$ follows from Lemma \ref{lem:nonemptycon}). Hence, the bounds found in Theorem \ref{maintheorem} may be improved, which in turn could shrink the bounds of $\chi(mk,k)$ in Propositions \ref{prop:4.2new} and \ref{prop:4.3new}.

Another route the interested reader could take is to generalize the results of this paper to finding subsets of $\mathbb{Z}_{mk}$ that do not contain any $nk$-term cyclic arithmetic progressions mod $mk$, where $m,n,k \in \mathbb{Z}^+$ with $m > n$ and $nk \geq 3$; that is, bounding $b(x,y)$ where $x$ and $y$ share a common factor, rather than $x$ being a direct multiple of $y$. While such results will not enhance the bounds of $W_{c}(k,r)$, they may lead to new discoveries, examples including bounding chromatic numbers $\chi(mk,nk)$ (defined similarly to $\chi(mk,k)$) of $\mathbb{Z}_{mk}$ and independence numbers of cyclic Van der Waerden hypergraphs $H_{mk,nk}$.

To start, the author conjectures that we may define $D(mk,nk)$, the set of common differences to consider when finding subsets of $\mathbb{Z}_{mk}$ that do not contain any $nk$-term cyclic arithmetic progressions mod $mk$, in a similar fashion to Lemma \ref{prop:2.1}; that is, the author conjectures the following:

\begin{conjecture}\label{conj:5.1}
    $D(mk,nk) = \{1 \leq g \leq m : g \mid nk\}$.
\end{conjecture}

\section*{Acknowledgements} The author would like to thank Peter Johnson, mentor during the 2023 NSF-funded summer REU program at Auburn University, for introducing the problems studied in this paper and for providing valuable guidance, particularly with Section \ref{sec: lemmas}. The author would also like to thank Darij Grinberg and Louis DeBiasio for their insightful feedback.

\end{document}